\documentclass[a4paper,12pt]{article}
\usepackage[utf8]{inputenc}
\usepackage[cm]{aeguill}
\usepackage{amsfonts}
\usepackage{latexsym}
\usepackage{graphicx}

\usepackage{amsmath}
\usepackage{amssymb}
\usepackage{amsthm}
\usepackage{bbm}
\usepackage{nicefrac}
\usepackage[T1]{fontenc}
 
\newtheorem{df}{Definition}[section]
\newtheorem{lm}[df]{Lemma}

\newtheorem{Th}[df]{Theorem}
\newtheorem{co}[df]{Corrolary}
\newtheorem{rem}[df]{Remark}

\newtheorem{OpenProblem}[df]{Open Problem}
\newcommand{\e}{\varepsilon}

\newcommand{\R}[1]{\mathbb{R}^{#1}}

\renewcommand{\d}{\operatorname{d}}

\newcommand{\epi}{\operatorname{Epi}}

\newcommand\INT[1]{\mathaccent'27{#1}}
\author{A. Fathi, M. Zavidovique}
\title{Ilmanen's Lemma on Insertion of C{$^{1,1}$} Functions\footnote{Work supported by ANR KAM faible}}

\begin{document}

\maketitle

\begin{abstract} We give a proof of Ilmanen's lemma, which asserts that between a locally semi-convex and a
locally semi-concave function it is possible to find a C$^{1,1}$ function.

\end{abstract}
\section*{Introduction}
This paper is essentially expository in nature. We will give a direct simple proof of the following fact due to Ilmanen, see \cite{Ilm}: between a locally semi-convex function (with a linear modulus) $f$ 
and  a locally semi-concave function (with a linear modulus) $g$, with $f\leq g$ everywhere, we can insert a C$^{1,1}$ function,
i.e. we can find a C$^{1,1}$ function $h$ with $f\leq h\leq g$ everywhere, see Theorem \ref{ins} below. This problem is of course trivial when $f<g$ everywhere,
and in that case one can find $h$ of class C$^\infty$. In his paper \cite{Ilm} Ilmanen sketches two proofs. There have been since then
other proofs, see for example \cite{Car}. About the same time we obtained our proof, Patrick Bernard \cite{Ber} gave a very nice one using iterated Lasry-Lions regularization.

We give a proof via a result on smoothness of lower convex envelops of  coercive functions, see Theorem \ref{sup}. This result is due to  Kirchheim and Kristensen \cite{KiK}, which itself builds up on  the work of Griewank and Rabier \cite{GrR}. The authors of the present work discovered independently this fact, and afterwards were able to point up at the earlier work. It is quite strange that our scheme to obtain easily 
Ilmanen's Lemma from \cite{KiK} did not materialize earlier. For sake of completeness we will derive a result on separation of closed subsets in manifolds that is useful in viscosity theory of PDE's, see for example \cite{Car}, from where we take the essential part of the argument.

The authors would like to thank Pierre Cardaliaguet. Without his help and stimulation this work would have never seen light.
The first author wishes to thank Naoki Yamada and Fukuoka University, for their support during the final writing of this paper.

\section{Semi-concave functions}
For functions defined on an open subset of a Euclidean space, we briefly recall 
the definition and properties of a locally semi-concave (or semi-convex) function 
for a modulus $\omega$.

Here by modulus $\omega$ we mean a {\sl continuous non-decreasing concave} function $\omega :[0,+\infty[\to [0,+\infty[$, with $\omega(0)=0$.
A linear modulus is a modulus $t\mapsto kt$, where $k\geq0$. 

In this work, if $O\subset\R{n}$, a function $f:O\to \R{}$ is said to be semi-concave with modulus $\omega$ if we can find a constant $C<+\infty$, such that for for every $x\in O$ we can find a linear map 
$\ell_x:\R{n}\to\R{}$ satisfying
$$\forall y\in O,\ \  f(y)-f(x)\leq \ell_x(y-x)+C\lVert y-x\rVert\omega(\lVert y-x\rVert).$$
We say that $f$ is semi-convex with modulus $\omega$ if $-f$ is semi-concave with modulus $\omega$.
Note that a concave (resp.\ convex) function is semi-concave (resp.\ semi-convex) for any modulus $\omega$.
 The function $f:O\to \R{}$ is  locally semi-concave (resp.\ semi-convex) with modulus $\omega$, if for every $x\in O$, we can find an open neighborhood $V_x$ of $x$ such that the restriction
$f\vert V_x$ is semi-concave on $V_x$ (resp.\ semi-convex) with modulus $\omega$. 

More generally, a function $f:O\to\R{}$ is said to be locally semi-concave (resp.\ convex) if for each $x\in O$, we can find an open neighborhood $V_x$ of $x$, and a modulus $\omega_x$, such that the restriction
$f\vert V_x$ is semi-concave on $V_x$ (resp.\ semi-convex) with modulus $\omega_x$. 

A standard reference for semi-concave functions is the book \cite{CaS}. A well adapted treatment for this work is Appendix A of 
\cite{FaF}. Note however that the definition given here of a semi-concave function with modulus $\omega$  is slightly more general than the one given in \cite{FaF}. In fact what we call here a semi-concave function with modulus $\omega$, is a function for which there exists $C$ such that the function is semi-concave  with modulus $C\omega$ in the sense of \cite{FaF}. This does not make significant difference but it allows to simplify somewhat the statements. 

We make one more observation. If $\omega :[0,+\infty[\to [0,+\infty[$ is continuous non-decreasing, and concave, then
$$\forall \lambda \geq 0,\forall t\geq 0, \ \  \omega(\lambda t)\leq \max(1,\lambda)\omega (t).$$
Hence, if $\lambda \geq 0$ is fixed, with the definition given any function (locally) semi-concave with modulus $t\mapsto\omega(\lambda t)$ is also (locally) semi-concave with modulus $\omega$.

We know recall that a function $f:O\to \R{}$ is said to be C$^{1,\omega}$, if it is differentiable everywhere, and for every $x\in O$, we can find a neighborhood $V_x$ and a constant $K_x$ such that
$$\forall y,z\in V_x,\ \  \lVert \d_yf-\d_zf\rVert\leq K_x\omega(\lVert y-z\rVert).$$
If the modulus $\omega$ is linear a C$^{1,\omega}$ is simply a C$^{1,1}$ function.

An extremely useful fact is that a function $f:O\to \R{}$ is both locally semi-concave and locally semi-convex for the modulus $\omega$ if and only if it is C$^{1,\omega}$. A proof of this fact can be found for example in \cite[A19, page 36]{FaF}.
Conversely, any C$^{1,\omega}$ is both locally semi-convex and concave for the modulus $\omega$.

As explained in \cite{FaF} the notions of locally semi-concave or semi-convex functions (for a linear modulus) make perfect sense on a manifold.

\section{Convex envelop of a coercive function}
In this section, we state and prove a result due to Kirchheim and Kristensen \cite{KiK}. We will restrict ourselves to the case where the ambient space is an open subset $O\subset \R{n}$, for some $n\in \mathbb{N}$. In order to avoid problems coming from the boundary, we will only consider coercive functions defined on $O$.
\begin{df}\rm
Let $f:O\to \R{}$. We will say that $f$ is coercive on $O$,  if for any $r\in \R{}$ the set $f^{-1}(]-\infty,r])$ is relatively compact in $O$. \end{df}

\begin{rem}\label{rem}\rm 1) A  continuous function $f:O\to \R{}$ is coercive on $O$ if and only  if it is bounded from below, and proper, i.e. the set $f^{-1}(K)$ is compact, for every compact subset $K\subset\R{}$.

2) A  function $f:O\to \R{}$ is coercive on $O$ if and only if $f(x)\to+\infty$, as $x\to +\infty$ in the locally compact space $O$.

3) For a function $f:\R{n}\to\R{}$ coercivity is therefore the usual concept $f(x)\to+\infty$, as $\lVert x\rVert\to+\infty$.

4) If $O$ is a bounded open set, then a coercive  function on $O$ is a function $f:O\to \R{}$ which verifies that
$$\lim_{x\to \R{n}\setminus O}\|f(x)\|= +\infty.$$
\end{rem}

We are going to prove the following result, see \cite[Theorem page 726]{KiK}:
\begin{Th}\label{sup}
Suppose $O$ is an open convex subset in $\R{n}$. If $f:O\to \R{}$ is coercive on $O$ and locally semi-concave with modulus $\omega$, then its (lower) convex envelop $f^*$ is C$^{1,\omega}$.
\end{Th} 
\begin{proof} The proof is also taken from \cite{KiK}. 
Note that $f^*$ being convex is locally semi-convex with modulus $\omega$ (in fact with any modulus). From
a result quoted above, it therefore suffices to show that $f^*$ is also  locally semi-concave with modulus $\omega$.

Since $f$ is bounded below, subtracting a constant if necessary, we can assume $f\geq 0$ everywhere. We now recall that the epigraph of $f$ is
$$\epi(f)=\{(x,t)\mid x\in 0, t\geq f(x)\}.$$
As is well-known, the convex closure of $\epi(f)$ in $O\times \R{}$ is the epigraph of $f^*$.
Since $O\times\R{}\subset \R{n}\times\R{}=\R{n+1}$, we can  apply Caratheodory's Theorem, 
see  \cite[Theorem 17.1, page 155]{Roc} to obtain that any point in the convex set generated by $\epi(f)$ is a convex combination of $n+2$ points in $\epi(f)$ (in fact the number can be cut down to $n+1$ using instead \cite[Corollary 17.1.5, page 157]{Roc}, but this is not essential). This implies that, if  $x\in O$ is given,
we can find sequences
$x^m_1, \dots, x^m_{n+2}\in O, \alpha^m_1,\dots,\alpha^m_{n+2}\in [0,1], m\geq1$ such that 
$$ \alpha^m_1\geq\cdots\geq\alpha^m_{n+2},\sum_{i=1}^{n+2} \alpha^m_i=1,\ \  \sum_{i=1}^{n+2} \alpha^m_ix^m_i=x $$
and
\begin{equation}\label{f*x}
\sum_{i=1}^{n+2}\alpha^m_if(x^m_i)\to f^*(x), \text{when  $m\to\infty$.}
\end{equation}

We now fix an arbitrary compact subset $K\subset O$. Set $M_K=\sup\{f(z)\mid z\in K\}<+\infty$. Since $f$ is coercive
the set 
$$K'=\{z\in O\mid f(z)\leq (n+2)(M_K+1)\}$$
 is compact.
Since $f$ is locally semi-concave with modulus $\omega$ and $K'$ is compact, we can find positive constants $\delta$ and $\lambda$ such that the inclusion $\overline B(y,\delta)\subset O$ holds for every $y\in K'$, and
\begin{equation}\label{sem}
\forall y\in K',\exists \ell_y\in \R{n*},\forall v\in B(0,\delta),\ \ f(y+v)\leqslant f(y)+\ell_y(v)+\lambda \lVert v\rVert\omega(\lVert v\rVert),
\end{equation}
where $ \R{n*}$ is the dual of $ \R{n}$, i.e.\ the set of linear maps from $ \R{n}$ to $\R{}$.

Let us now assume $ x\in K$. We consider the sequences
$x^m_1, \dots, x^m_{n+2}\in O, \alpha^m_1,\dots,\alpha^m_{n+2}\in [0,1], m\geq1$ 
obtained above. Since $\sum_{i=1}^{n+2}\alpha^m_if(x^m_i)\to f^*(x)\leq f(x)\leq M_K$, dropping the first terms we can assume
$$\forall m\geq 1, \ \ \sum_{i=1}^{n+2}\alpha^m_if(x^m_i)\leq M_K+1.$$
Note that as we said in the beginning of the proof, we are assuming (without loss of generality) that $f\geq 0$.
Therefore
\begin{equation}\label{MK1}
\forall m\geq 1, \ \ \alpha^m_1f(x^m_1)\leq M_K+1.
\end{equation}
We now use the assumed inequalities $\alpha^m_1\geq\cdots\geq\alpha^m_{n+2}$, to obtain $(n+2)\alpha_1^m\geq \sum_{i=1}^{n+2} \alpha^m_i=1$, hence 
\begin{equation}\label{alpha1}
(n+2)\alpha_1^m\geq 1.
\end{equation}
Together with (\ref{MK1}), this yields
\begin{equation}\label{MK2}
\forall m\geq 1,\ \  f(x^m_1)\leq (n+2)(M_K+1).
\end{equation}
Therefore the sequence $x^m_1$ lies in the compact set $K'$. Extracting we can assume 
$$x^m_1\to x_1\in K',\ \  \alpha^m_1\to \alpha_1.$$
It follows  from (\ref{alpha1}) that we also have
 $$(n+2)\alpha_1\geq 1. $$
 Therefore, if $\lVert h\rVert \leq \delta/(n+2)$,
we get $\lVert \alpha_1^{-1}h\rVert\leq \delta$, and by the choice of $\delta$ above, we have 
$x^m_1+\alpha_1^{-1}h,x_1+\alpha_1^{-1}h\in O$. Therefore 
\begin{equation}\label{f*y}
f^*\left(\alpha^m_1(x_1+\alpha_1^{-1}h)+\sum_{i=2}^{n+2} \alpha^m_ix^m_i\right)\leq \alpha^m_1f(x^m_1+\alpha_1^{-1}h)+\sum_{i=2}^{n+2} \alpha^m_if(x^m_i).
\end{equation}
Note now that $\alpha^m_1(x_1+\alpha_1^{-1}h)+\sum_{i=2}^{n+2} \alpha^m_ix^m_i=\alpha^m_1\alpha_1^{-1}h+\sum_{i=1}^{n+2} \alpha^m_ix^m_i=x+\alpha^m_1\alpha_1^{-1}h\to x+h$, as $m\to +\infty$. Since the function $f^*$ is convex, it is therefore continuous on the open set $O$.
Together with (\ref{f*x}) and (\ref{f*y}), this yields
\begin{align*}
f^*(x+h)-f^*(x)&=\lim_{m\to\infty}f^*\left(\alpha^m_1(x_1+\alpha_1^{-1}h)+\sum_{i=2}^{n+2} \alpha^m_ix^m_i\right)-\sum_{i=1}^{n+2}\alpha^m_if(x^m_i)\\
&\leq \lim_{m\to\infty}\alpha^m_1f(x^m_1+\alpha_1^{-1}h)+\sum_{i=2}^{n+2} \alpha^m_if(x^m_i)-\sum_{i=1}^{n+2}\alpha^m_if(x^m_i)\\
&=\lim_{m\to\infty}\alpha^m_1f(x^m_1+\alpha_1^{-1}h)-\alpha^m_1f(x^m_1)\\
&=\alpha_1[f(x_1+\alpha_1^{-1}h)-f(x_1)].
\end{align*}
Since $x_1\in K'$, and $\lVert \alpha_1^{-1}h\rVert\leq \delta$, we can now use
(\ref{sem}) to obtain a $\ell_{x_1}\in \R{n*}$, such that
\begin{align*}
f^*(x+h)-f^*(x)&\leq \alpha_1[\ell_{x_1}(\alpha_1^{-1}h)+\lambda \lVert \alpha_1^{-1}h\rVert\omega(\lVert \alpha_1^{-1}h\rVert)]\\
&=\ell_{x_1}(h)+\lambda \lVert h\rVert\omega(\alpha_1^{-1}\lVert h\rVert).
 \end{align*}
 Since $\omega$ is concave and non-decreasing, and $\alpha_1^{-1}\leq n+2$, we conclude that for all $x\in K$, we can find
 $\ell_{x_1}\in \R{n*}$, such that for all $h\in \R{n}$ with $\lVert h\Vert \leq \delta/(n+2)$, we have
 $$f^*(x+h)\leq f^*(x)+ \ell_{x_1}(h)+(n+2)\lambda \lVert h\rVert\omega(\lVert h\rVert).$$
 This shows that $f^*$ is locally semi-concave  with modulus $\omega$ in a neighborhood 
of $K$. Since $K$ is an arbitrary compact subset of $O$, the convex function $f^*$ is locally semi-concave on $O$ with modulus $\omega$. It is therefore C$^{1,\omega}$.
 \end{proof}
\begin{rem}\rm The condition of coercivity is not an artificial one. Here is an interesting example.
Consider the square $[0,1]\times [0,1]$. We define a function $f: [0,1]\times [0,1]\to [0,1]$ affine on the two triangles
$T_\geq=\{(x,y)\mid x,y\in [0,1], x\geq y\},T_\leq=\{(x,y)\mid x,y\in [0,1], x\leq y\}$ with
$f(0,0)=f(1,1)=1$ and $f(0,1)=f(1,0)=0$, see figure (\ref{pict}). 
It is concave, and hence semi-concave for any modulus. 
Its graph in $\R{3}$ is the upper part of the tetrahedron spanned by the four points $(0,0,1),(1,1,1),(0,1,0),(1,0,0)$. Therefore its lower convex envelop $f^*$ is the lower part of the tetrahedron. Hence $f^*$ is affine on each of the two triangles $T^\geq=\{(x,y)\mid x,y\in [0,1], x+y\geq 1\},T^\leq=\{(x,y)\mid x,y\in [0,1], x+y\leq 1\}$. This function $f^*$ is not differentiable at any point of the diagonal $\Delta=\{(x,y)\mid x,y\in [0,1], x= y\}$.
\end{rem}

\begin{figure}\label{pict}
\begin{center}
\includegraphics[height=120mm]{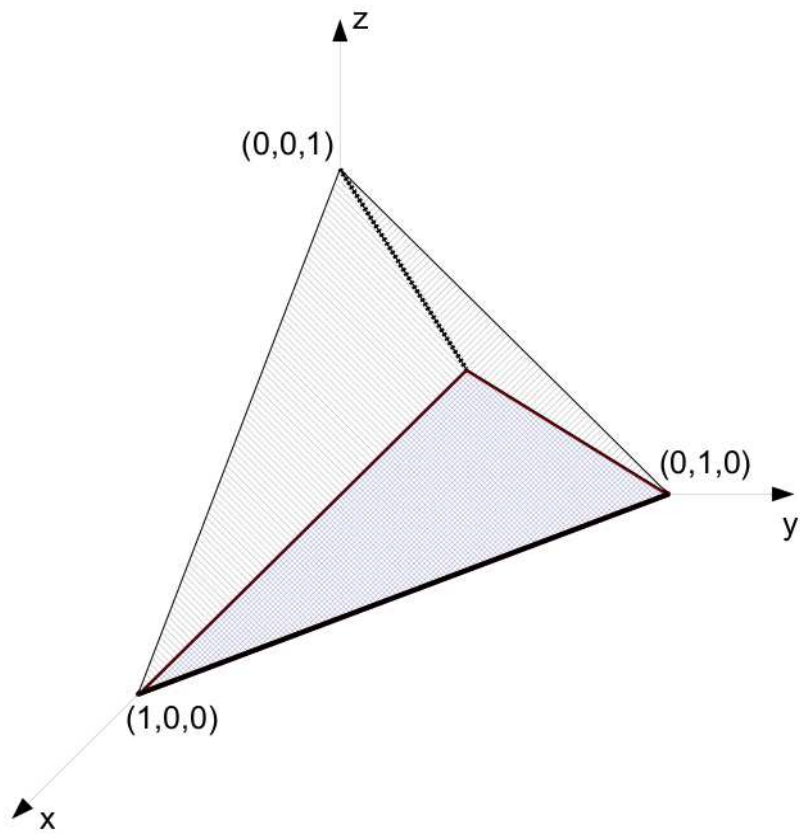}
\caption{A counter example in the non super-linear case}
\end{center}
\end{figure}

\section{Ilmanen's insertion lemma}
In this section, we prove the main theorem of this article which is that between a locally semi-concave function and a locally semi-convex function both for a linear modulus, there is a C$^{1,1}$ function. More precisely, we prove the following:
\begin{Th}[Ilmanen's Insertion Lemma]\label{ins}
 Let $M$ be a C$^2$ manifold, and let us consider $f:M\to \R{}$  a locally semi-convex function for a linear modulus and
 $g:M\to\R{}$ a locally semi-concave function for a linear modulus. If $f\leqslant g$ then there exists a C$^{1,1}$
 function $h:M\to\R{}$ such that $f\leqslant h\leqslant g$. 
\end{Th}
The proof of this theorem is mainly local, it is enough to prove that for each $x\in M$ there is a neighborhood of $x$, $V_x$ and a $C^{1,1}$ function $h_x:V_x\to \R{}$ such that $f\leqslant h_x\leqslant g$ on $V_x$. As a matter of fact, if $(\varphi_x)_{x\in M}$ is a $C^\infty$ partition of unity subordinated to $(V_x)_{x\in M}$ then the function 
$$h=\sum_{x\in M} \varphi_x h_x$$
clearly satisfies the requirements of \ref{ins}. \\
>From the discussion above, it is enough to prove the following:
\begin{lm}
 Suppose $B(x,R)$ is the open Euclidean ball centered in $x\in \R{n}$ of radius $R<+\infty$. If $f,g:B(x,R)\to \R{}$ are respectively semi-convex and semi-concave for a linear  modulus, and if moreover $f\leqslant g$, then we can find a C$^{1,1}$ function $h:B(x,R)\to \R{}$ with $f\leqslant h\leqslant g$.
\end{lm}
\begin{proof} We will denote by  $\langle \cdot ,\cdot \rangle$ and by $\|\cdot \|$ the usual scalar product and Euclidean norm on $\R{n}$. 
Because $f$ is assumed semi-concave (rather than just locally semi-concave) with a linear modulus, we can find a constant $K\geq 0$, such that for every $y\in B(x,R)$, we can find $\ell_y\in \R{n*}$ such that 
$$\forall z\in B(x,R),\ \  f(z)-f(y)\geq \ell_y(z-y)-K\lVert z-y\rVert^2.$$
Therefore, if we define the function $F$  on $B(x,r)$ by
$$F_1(z)=f(z)+\frac12 K\|z\|^2,$$
a simple computation yields
$$\forall z\in B(x,R), \ \ F_1(z)-F_1(y)\geq \ell_y(z-y)+\langle y ,z-y \rangle.$$
Therefore $F_1$ is convex on $B(x,R)$. Since $B(x,R)$ is a bounded convex set, we also obtain that $F_1$ is bounded from below.
Using that the function $y\mapsto 1/(R^2-\|y-x\|^2)$ is convex and coercive on $B(x,R)$, we conclude that the function
$F: B(x,R)\to \R{}$ defined by
$$F(y)=F_1(y)+\frac{1}{R^2-\|y-x\|^2}=f(y)+\frac12 K\|y\|^2+\frac{1}{R^2-\|y-x\|^2}$$
is convex and coercive on $B(x,R)$. We now define the function $G$ by
$$\forall y\in B(x,r), \ \ G(y)=g(y)+\frac 12 K\|y\|^2+\frac{1}{R^2-\|y-x\|^2}.$$
This function $G$ is locally semi-concave for a linear modulus, because it is the sum of the semi-concave function for a linear modulus $g$ and of a C$^{\infty}$ function. Since $F$ is convex, coercive, and $F\leq G$, it follows that $G$ is coercive and $F\leq G^*$, where $G^*$ is the lower convex envelop of $G$. We get $F\leqslant G^*\leqslant G$. By Theorem \ref{sup}, this convex envelop $G^*$ is C$^{1,1}$. 
It remains to set
$$\forall y\in B(x,r),\ \  h(y)=G(y)-\frac12 K\|y\|^2-\frac{1}{R^2-\|y-x\|^2},$$
then $f\leqslant h\leqslant g$, and $h$ is again C$^{1,1}$ as the sum of a C$^{1,1}$ function and a C$^\infty$ function.
\end{proof}
As an easy corollary we obtain the following more precise statement:
\begin{co}\label{coril} Under the hypothesis of Theorem \ref{ins} above,
denote by $E$ the set
$$E=\left\{x\in M, f(x)=g(x)\right\}.$$ 
 There exists a C$^{1,1}$ function $h:M\to\R{}$ such that $f\leqslant h\leqslant g$ and such that $f<h<g$ on $M\setminus E$.
 \end{co}
 \begin{proof}
 The function $\Phi=(g-f)/3$ is by definition a continuous function which is non negative and vanishes exactly on $E$.  Consider a covering of $M\setminus E$ by  open sets 
 $$M\setminus E=\bigcup_{i\in A}O_i,$$
such that each open set has a compact closure $\overline O_i$ contained in  $M\setminus E$. We choose $(\phi_i)_{i\in A}$ a partition of unity subordinated to the covering $(O_i)_{i\in A}$.
Note that each $\phi_i$ has compact support disjoint from $E$.
It follows that we can extend it to a C$^\infty$ function on $M$, with $\phi_i\equiv 0$ on $E$. Finally, for each $i\in A$, set $\alpha_i=\inf_{O_i} \Phi$. Note that each $\alpha_i$ is positive because $\overline O_i$ is compact and contained in $M\setminus E$. If $0<\beta_i\leq \alpha_i$, the function 
\begin{equation}\label{sum}
\phi=\sum_{i\in A}\beta_i \phi_i
\end{equation}
is continuous and $>0$ on $M\setminus E$, obviously $0$ on $E$. It also satisfies 
$$\forall x\in M,0\leqslant \phi(x)\leqslant \Phi(x)$$
therefore $\phi$ is continuous on $M$.  Moreover, up to taking $\beta_i$ rapidly decreasing to $0$,
see for example \cite[Lemma 3.2, page 722]{par}, we can assume without loss of generality that 
the sum (\ref{sum}) is convergent for the compact open C$^\infty$ topology on C$^\infty(M,\R{})$,
and therefore that the function $\phi$ is C$^\infty$.

Now let us consider the functions $F=f+\phi$ and $G=g-\phi$. These functions still verify the hypothesis of \ref{ins} therefore there is a function $\varphi$ between $F$ and $G$, but since $\phi$ is positive on $M\setminus E$, this proves the corollary.
\end{proof}

\section{Geometric form}
If $(X,d)$ is a metric space, as usual for a non-empty subset $C\subset X$, and $x\in X$,
we define
$$d(x,C)=\inf\{d(x,c)\mid c\in C\}.$$
As is well known, the function $x\mapsto d(x,C)$ is continuous (and even Lipschitz). Moreover, we have
$$d(c,C)=d(x,\overline C),\text{ and }\overline C=\{x\in X \mid d(x,C)=0\},$$
where $\overline C$ is the closure of $C$.

If $A,B\subset X$, we define
$$d(A,B)=\inf\{d(x,y)\mid x\in A,y\in B\}=\inf_{x\in A}d(x,B)=\inf_{y\in B} d(y,A).$$
We will need the following lemma which we state and prove for a Riemannian manifold. From the proof it is clear that it holds for metric length  spaces. Note also that this lemma does not hold for a general metric space.
\begin{lm}\label{lemmedist} Suppose that the distance $d$ on the connected manifold $M$  comes from a Riemannian metric.
Let $A,B$ and $S$ be non-empty subsets of $M$, with $A\subset \overline S$, and $B\cap \INT{S}=\varnothing$, then
$$d(A,B)\geq d(A,\partial S)+d(\partial S,B),$$
where $\partial S$ is the boundary of $S$ in $M$. (Note that $\partial S$ is not empty since the non-empty set $B$ is contained in 
$M\setminus \INT{S}$.)
\end{lm}
\begin{proof} Since the distance $d$ is obtained from Riemannian length of curves, we can find a sequence of smooth curves $\gamma_n:[0,1]\to M$, such that $\gamma_n(0)\in A, \gamma_n(1)\in B$,
and $\operatorname{length}(\gamma_n)\to d(A,B)$. For every $n$, we have $\gamma_n([0,1])\cap \partial S\neq \varnothing$.
In fact, if this were not true, for some $n$, we would have that the connected set  $\gamma_n([0,1])$ is contained in
the disjoint union of open subsets $\INT{S}\cup (M\setminus \overline S)=M\setminus \partial S$.
This would imply that $\gamma_n([0,1])$ is included in exactly one of the two sets.
This is impossible, because $\gamma_n(0)\in A$, therefore $\gamma_n(0)\notin M\setminus \overline S$,
and $\gamma_n(1)\in B$, therefore $\gamma_n(1)\notin \INT{S}$.

Since $\gamma_n([0,1])\cap \partial S\neq \varnothing$, for each $n$, we can find 
$t_n\in [0,1]$ such that $\gamma_n(t_n)\in \partial S$. We have
\begin{align*}
\operatorname{length}(\gamma_n)&\geq d(\gamma_n(0),\gamma_n(t_n))+d(\gamma_n(t_n),\gamma_n(0))\\
&\geq d(A,\partial S)+d(\partial S, B)
\end{align*}
It suffices to let $n\to \infty$ to finish the proof.
\end{proof}
This lemma has several consequences. To obtain them, we recall the following facts.
If $C\neq \varnothing$ is a subset of the metric space $X$, and $r> 0$, we set
$$V_r(C)=\{x\in X\mid d(x,C)<r\}.$$
Obviously, the set  $V_r(C)$ is open in $M$, and its closure $\overline V_r(C)$ is contained in the closed subset 
$\{x\in X\mid d(x,C)\leq r\}$.
This implies that 
$$\partial V_r(C)\subset \{x\in X\mid d(x,C)=r\}.$$
In particular, if $\partial V_r(C)\neq\varnothing$, then
$$d(\partial V_r(C),C)=r.$$
For a general metric, the inclusions $\overline V_r(C)\subset \{x\in X\mid d(x,C)\leq r\}, \partial V_r(C)\subset \{x\in X\mid d(x,C)=r\}$ are
usually strict. We will now obtain as a consequence of Lemma \ref{lemmedist} that they are equality for the case of Riemannian manifolds.
\begin{co} Suppose $M$ is a connected manifold endowed with a distance $d$ coming from a Riemannian metric.
Let $C\subset M$ be a non-empty subset, and let $r$ be a $>0$ number.
We have 
\begin{itemize}
\item[\rm(1)]  if $x\notin \INT{C}$, then $d(x, C)=d(x,\overline C)=d(x, \partial C)$.
\item[\rm (2)] for every $x\notin V_r(C)$, we have $d(x,C)=d(x,\partial V_r(C))+r$;
\item[(3)] $\overline V_r(C)= \{x\in X\mid d(x,C)\leq r\}$, and $\partial V_r(C)=\{x\in X\mid d(x,C)=r\}$;

\end{itemize}
Moreover if $A,B$ are non-empty subsets of $M$, 
\begin{itemize}
\item[\rm (4)] for every $r$ such that $0<r\leq d(A,B)$, we have 
$$d(B,V_r(A))= d(B,\overline V_r(A))=d(B,\partial V_r(A)),$$
and 
$$d(A,B)=r+d(B,V_r(A))=r+d(B,\partial V_r(A)).$$
\end{itemize}
\end{co}
\begin{proof} To prove (1), we apply Lemma \ref{lemmedist}, $A=S=C, B=\{x\}$ to obtain
$$d(x,C)\geq d(x, \partial C)+ d(\partial C,C).$$
But $d(\partial C,C)=0$. Therefore, we obtain $d(x,C)\geq d(x,\partial C)$. Note that the opposite inequality is true, since
$d(x,C)=d(x,\overline C)$, and $\partial C\subset \overline C$.

For (2), notice that $C\subset V_r(C)$, and $\{x\}$ is disjoint from the open subset $V_r(C)$.
Therefore, we can apply Lemma \ref{lemmedist}, with $A=C,S=V_r(C), B=\{x\}$ to obtain
$$d(x,C)\geq d(C, \partial V_r(C))+d(x, \partial V_r(C)).$$
But, as we have noticed above $d(C, \partial V_r(C))=r$. Hence, we obtain $d(x,C)\geq d(C, \partial V_r(C))+r$
to prove the converse inequality, we consider $y\in V_r(C)$, and write
$$d(x,C)\leq d(x,y)+d(y,C)\leq d(x,y)+r.$$
Taking the infimum over $y\in \partial V_r(C)$, we obtain $d(x,C)\leq d(C, \partial V_r(C))+r$

For (3), since $\overline V_r(C)=V_r(C)\cup \partial V_r(C)$, and $\partial V_r(C)\subset\{x\in X\mid d(x,C)=r\}$, 
we only need to show that $\partial V_r(C)\supset\{x\in X\mid d(x,C)=r\}$. If $d(x,C)=r$, we have $x\notin V_r(C)$,
and by (2) above $d(x,\partial V_r(C))=0$.
Since $\partial V_r(C))$ is closed, we obtain $x\in \partial V_r(C)$.

To prove (4) we notice that for any $y\in B$, we have $y\notin V_r(A)$, hence by  (1) above we obtain
$d(y,V_r(A))=d(y,\partial V_r(A))$.
Taking the infimum over $y\in B$ yields $d(B,V_r(A))=d(B,\partial V_r(A))$
Moreover, by (2), for $y\in B$, we get $d(y,A)=d(y,\partial V_r(A))+r$. Taking the infimum over $y\in B$ yields 
$d(B,A)=r+d(B,\partial V_r(A))$.
\end{proof}
The following theorem is basically the geometric form of Ilmanen's lemma, even if it is stated under somewhat different hypothesis.
\begin{Th}[Geometric Form of Ilmanen's Lemma]\label{GeomFormIlmanen}
Suppose $M$ is a connected manifold, endowed with a Riemannian metric of class C\/$^2$. We 
denote by $d$ the distance associated to the Riemannian metric.
If  the closed non-empty subset $A$ of $M$, and $a>0$ are such that $V_a(A)\neq M$ (or equivalently $\partial V_a(M)\neq \varnothing$ by the connectedness of $M$), then for every $\rho\in ]0,a[$, we can find a closed domain $\Sigma_\rho$
such that 
\begin{itemize}
\item[\rm(i)]  the boundary $\partial \Sigma_\rho$ of $\Sigma_\rho$ is C\/$^{1,1}$ submanifold;
\item[\rm(ii)] $\Sigma_\rho\supset \overline V_\rho(A)$;
\item[\rm(iii)] $\Sigma_\rho\subset \{x\mid d(x, M\setminus V_a(A))\geq a-\rho\}\subset V_a(A)$;
\item[\rm(iv)] $d(\partial\Sigma_\rho,A)=\rho$ and $d(\partial\Sigma_\rho, M\setminus V_a(A))=a-\rho$.
\item[\rm(v)]$\{x\in M\mid d(x,A)=\rho, d(x, M\setminus V_a(A))=a-\rho)\}\subset \partial\Sigma_\rho$.
\end{itemize}
\end{Th}
\begin{proof} Note that both $A$ and $M\setminus V_a(A)$ are not empty. We also have
$d(A, M\setminus V_a(A))\geq a$, Moreover, since $\partial \overline V_a(A)\neq \varnothing$, we can find an $x_0$ such that $d(x_0,A)=a$.
This $x_0$ is in  $M\setminus V_a(A)$. Hence, we get
$$d(A, M\setminus V_a(A))=a.$$
In particular, we obtain
\begin{equation}\label{ine}
\forall x\in M, \ \ d(x,A)+d(x, M\setminus V_a(A))\geq d(A, M\setminus V_a(A))=a.
\end{equation} 
Therefore, if we define $f,g:M\to \R{}$ by
$$ f(x)=a-d(x, M\setminus V_a(A)),\text{ and, }g(x)=d(x,A),$$
we have $f\leq g$ everywhere. If $B$ is a closed subset of $M$, as is well-known $x\mapsto d(x,B)$ is a viscosity solution
on $M\setminus B$ of the eikonal equation
\begin{equation}\label{eikonal}
\lVert \d_xu\rVert_x=1,
\end{equation}
where $\lVert \cdot\rVert_x$ is the norm obtained from the Riemannian metric on the cotangent space $T^*_xM$ at $x$. 
Since we are assuming the Riemannian metric to be C$^2$, this implies that $x\mapsto d(x,B)$ is locally semi-concave
with linear modulus on $M\setminus B$. We obtain from this that the restrictions of $f$ and $g$ to the open set
$U_a=V_a(A)\setminus A$ are respectively locally semi-convex and locally semi-concave with a linear modulus.
Therefore, by Corollary \ref{coril}, we can find a C\/$^{1,1}$ function $h:U_a\to \R{}$ with $f\leq h\leq g$, and
$f<h<g$ outside of the set 
$$E=\{x\in U_a\mid f(x)=g(x)\}=
\{x\in U_a\mid d(x,A)+d(x, M\setminus V_a(A))=a\}.$$
In fact, if $\rho$ was a regular value of $h$, we could finish with $\Sigma_\rho=A\cup \{x\in U_a\mid h(x)\leq \rho\}$.
Unfortunately, a C\/$^{1,1}$ function may not even have regular values if the dimension of $M$ is $\geq 3$.

We now proceed to modify $h$ in order to have $\rho$ as a regular value. We will need to use the  strong
topology on space of differentiable maps. This topology is also called the Whitney topology. On these matter, we refer to  \cite[Chapter 2, \S1]{Hir}. If $N,P$ are manifolds
we will use the notation $C^r_S(N,P)$, introduced in \cite{Hir},  for the space of C\/$^r$ maps from $N$ to $P$
endowed with the strong (or Whitney) topology.

Since $h\leq g$, with equality on $E$, and $g$ is a viscosity solution of the eikonal equation,
we obtain
$$\forall x\in E,\ \  \lVert \d_xh\rVert_x\geq 1.$$
(In fact, using $f\leq h$, we could show equality on $E$ in the above inequality, but we will not need this). In particular
$\d_xh\neq 0$ on the closed subset $E$ of $U_a$. Therefore we can find an open subset $W\subset U_a$ such that
the derivative of $h$ is never $0$ on the closure $\overline W$ in $U_a$. Note that this imply that we can find a neighborhood
$\cal V$ of $h$ in $C^1_S(U_a,\R{})$ (the space  $C^1(U_a,\R{})$ endowed with Whitney or strong topology) such that for every 
$\tilde h\in {\cal V}$, we have $d_x\tilde h\neq0$, for every $x\in \overline W$. We now pick a function $\theta: U_a \to [0,1]$ such that $\theta\equiv 1$ on a neighborhood of $U_a\setminus W$, and
whose support $F$ is contained in $U_a\setminus E$. If $\tilde h:U_a\setminus E\to\R{}$ converges in the C\/$^1$ strong topology on $C_S^1(U_a\setminus E,\R{})$ to the restriction $h|U_a\setminus E$, then $\theta \tilde h$ converges to 
$\theta h$ in  $C^1_S(U_a\setminus E,\R{})$. Since all these functions are $0$ outside of the closed set $F$ which is disjoint from $E$, in fact, we obtain that $\theta\tilde h$ converges to 
$\theta h$ in $C_S^1(U_a,\R{})$. It follows that $\theta \tilde h+(1-\theta)h$ converges to $h$
in $C_S^1(U_a,\R{})$. Therefore we can find an open neighborhood $\cal W$ of $h|U_a\setminus E$
in $C_S^1(U_a\setminus E,\R{})$, such that for every $\tilde h\in{\cal W}$, we have 
$\theta \tilde h+(1-\theta)h\in{\cal V}$ and therefore $\theta\tilde h+(1-\theta)h$ has no critical point in $\overline W$.
Since $f<h<g$ on $U_a\setminus E$ cutting down on the neighborhood $\cal W$ of $h|U_a\setminus E$ in $C_S^1(U_a\setminus E,\R{})$, we may also assume that 
we have $f<\tilde h<g$, for every $\tilde h\in {\cal W}$. We now use the fact, see \cite[Exercise 2 (a), page 74]{Hir}, that C\/$^\infty$ functions with a given prescribed value as
a regular value are dense in the strong (or Whitney) topology to obtain a C\/$^\infty$ map $\tilde h \in {\cal W}$
 with $\rho$ as a regular value
on $U_a\setminus E$. Wet set $\overline h=\theta \tilde h+(1-\theta)h$. Since $\theta$ is equal to $1$ on a neighborhood
of $U_a\setminus W$, it follows that $\overline h=\tilde h$ on this neighborhood, and therefore $\rho$ is a regular value of $\overline h$
on a neighborhood of $U_a\setminus W$. Since $\tilde h\in {\cal W}$, we know that $\overline h$ has no critical point in $W$.
Therefore $\rho$ is a regular value of the C\/$^{1,1}$ function $\overline h$. Note that by construction, we have
$$f\leq \overline h\leq g$$
everywhere on $U_a$ (even with strict inequalities on $U_a\setminus E$).
This can be rewritten 
\begin{equation}\label{GRANDEINE}
\forall x\in V_a(A)\setminus A,\ \  a-d(x, M\setminus V_a(A))\leq \overline h\leq d(x,A).
\end{equation}
We now set 
$$\Sigma_\rho=A\cup\{x\in V_a(A)\setminus A\mid \overline h(x)\leq \rho\}.$$
>From the right hand side inequality in (\ref{GRANDEINE}), we obtain the point (ii) of the theorem
$$\overline V_\rho(A)\subset \Sigma_\rho.$$
Note that this implies that $\partial \Sigma_\rho$ is disjoint for $A$.
>From the left hand side inequality in (\ref{GRANDEINE}), taking into account that $d(A, M\setminus V_a(A))=a$, we obtain
that
$$\forall x\in \Sigma_\rho, d(x, M\setminus V_a(A))\geq a-\rho>0,$$
which is point (iii) of the theorem. Note that is implies that $\partial \Sigma_\rho$ is contained in the open set $V_a(A)$. Since it is also disjoint from $A$, we get $\partial \Sigma_\rho\subset V_a(A)\setminus A=U_a$. But 
$U_a\cap  \Sigma_\rho=\{x\in U_a\mid \overline h(x)\leq \rho\}.$
Since $\rho$ is a regular value of the C\/$^{1,1}$ map $\overline h:U_a\to\R{}$, the implicit function theorem implies that
$\overline h^{-1}(\rho)$ is a C\/$^{1,1}$ hypersurface and that $\partial \Sigma_\rho=\overline h^{-1}(\rho)\subset \Sigma_\rho$. This proves
that $\Sigma_\rho$ is closed and also  point (i) of the theorem.

It is clear now that point (v) of the theorem follows from the inequalities (\ref{GRANDEINE}).

It remains to prove point (iv) of the theorem. From (ii) and (iii), it follows that
$$\forall x \in \partial \Sigma_\rho,\ \  d(x,A)\geq \rho,\text{ and }d(x, M\setminus V_a(A))\geq a-\rho.$$
Therefore 
$$d(\partial\Sigma_\rho,A)\geq \rho, \text{ and }d(\partial\Sigma_\rho, M\setminus V_a(A))\geq a-\rho.$$
To finish the proof of (iv) it suffices to show that
$$d(\partial\Sigma_\rho,A)+d(\partial\Sigma_\rho, M\setminus V_a(A))\leq a=d(A, M\setminus V_a(A)).$$
But this follows Lemma \ref{lemmedist} . Since $A\subset \Sigma_\rho$, and the closed set $\Sigma_\rho$
is disjoint from $M\setminus V_a(A)$.
\end{proof}
>From Theorem \ref{GeomFormIlmanen}, we could obtain apparently stronger statements. We will only give this one:
\begin{Th}\label{dernier} Suppose $M$ is a connected manifold, endowed with a Riemannian metric of class C\/$^2$. We 
denote by $d$ the distance associated to the Riemannian metric.
If  $A,B\subset M$ are  closed non-empty disjoint subsets,  we can find a closed domain $\Sigma$ whose
boundary $\partial \Sigma$ is a C\/$^{1,1}$ submanifold, such that 
\begin{gather*}
A\subset \INT{\Sigma},\Sigma \cap B=\varnothing,\\
d(A,\partial \Sigma)=d(\partial \Sigma,B)=\frac{d(A,B)}2,\\
\left \{x\in M\mid d(x,A)=d(x,B)=\frac{d(A,B)}2\right\}\subset \partial \Sigma.
 \end{gather*}
Moreover, if $a=d(A,B)>0$, we can assume that $\Sigma\subset V_a(A)$. In particular, if $A$ is compact and the Riemannian metric
on $M$ is complete, we can assume also that $\Sigma$ is compact.
\end{Th}
\begin{proof} We first assume $d(A,B)=a>0$. We use Theorem \ref{GeomFormIlmanen}, with $A,a$ and $\rho=a/2$
 to obtain $\Sigma=\Sigma_{a/2}$. We have $A\subset \INT{\Sigma}\subset \Sigma\subset V_a(A)$, and also
 \begin{gather*}
 d(A,\partial \Sigma)=d(\partial \Sigma,M\setminus V_a(A))=\frac a2,\\
\left \{x\in M\mid d(x,A)=d(x,M\setminus V_a(A))=\frac a2\right\}\subset \partial \Sigma.
\end{gather*} 
  In particular, we have $\Sigma\cap B=\varnothing$, because $B\subset M\setminus V_a(A)$. This last inclusion implies that
\begin{equation}\label{lafin}
\forall x\in M,\ \  d(x,B)\geq d(x,M\setminus V_a(A)).
\end{equation}
Therefore $d(\partial \Sigma,B) \geq d(\partial \Sigma,M\setminus V_a(A))=a/2$. Since $d(A,\partial \Sigma)=a/2$,
to prove that $d(\partial \Sigma,B)=a/2$, it now suffices to show that
 $$d(\partial \Sigma,B)+d(A,\partial \Sigma)\leq a=d(A,B).$$
 but this follows from Lemma \ref{lemmedist}, since the closed set $\Sigma$ contains $A$, and is disjoint from $B$.
 
 We now show that any $x$ satisfying $d(x,A)=d(x,B)=a/2$ is necessarily in $\partial \Sigma$. By the definition of $\Sigma$
 it suffices to show $d(x,M\setminus V_a(A))=a/2$. By the inequality (\ref{lafin}), we know that $d(x,M\setminus V_a(A))\leq a/2$.
 therefore we get 
 $$a=\frac a2+\frac a2\geq d(x,A)+d(x,M\setminus V_a(A))\geq d(A,M\setminus V_a(A))\geq a.$$
 This implies that we have the equality $d(x,M\setminus V_a(A))= a/2$.
 This finishes the proof in the case $d(A,B)=a>0$.

 Suppose  $d(A,B)=0$. Since $A\cap B=\varnothing$, we can find a C$^\infty$ function $\varphi:M\to [0,1]$ such that
 $\varphi$ is $0$ on $A$ and $1$ on $B$. We pick a regular value $r\in ]0,1[$ of $\varphi$. The closed set 
 $\Sigma=\{x\in M\mid \varphi(x)\leq r\}$ has a C$^\infty$ boundary, contains $A$ and is disjoint from $B$. By Lemma
 \ref{lemmedist}, we obtain $0=d(A,B)\geq d(A,\partial \Sigma)+d(\partial \Sigma,B)$. It follows that $d(A,\partial \Sigma)=d(\partial \Sigma,B)=0$. Note that since $A,B$ are closed and disjoint the set  $\{x\in M\mid d(x,A)=d(x,B)=0=d(A,B)/2\}$ is empty.
 
 Suppose now that $A$ is compact, since it is disjoint from $B$, we must have $d(A,B)=a>0$. The $\Sigma$ constructed above
 is contained in $V_a(A)$. If the Riemannian metric is complete, then $V_a(A)$ is, like any bounded set, relatively compact.
 Therefore its closed subset $\Sigma$ is compact. 
 \end{proof}
 
 \begin{rem}\rm
 In the last part of the previous theorem (\ref{dernier}), even if the metric is not assumed to be complete, it is possible to find a compact set $\Sigma$ whose
boundary $\partial \Sigma$ is a C\/$^{1,1}$ submanifold, such that 
\begin{gather*}
A\subset \INT{\Sigma},\Sigma \cap B=\varnothing,\\
d(A,\partial \Sigma)+d(\partial \Sigma,B)=d(A,B),\\
\exists \e>0, \ \ \left \{x\in M\mid \e=d(x,A)=d(A,B)-d(x,B)\right\}\subset \partial \Sigma.
 \end{gather*}
In fact, it is enough to pick $\e>0$ small enough for the neighborhood $V_{2\e}(A)$ to be relatively compact in $M$ and to repeat the previous proof with $\rho=\e$.
 \end{rem}
 
\section{Open problem}
We would like to conclude the paper with an open problem:
\begin{OpenProblem} \rm Suppose that $\omega$ is a modulus, that $f$ and $g$ are respectively  a locally semi-convex semi-concave function for the modulus $\omega$, with $f\leq g$. Is it always possible to find a C$^{1,\omega}$ function $h$ with
$f\leq h\leq g$? What about the H\"older moduli $\omega_\alpha(t)=t^\alpha,\alpha<1$.
\end{OpenProblem}

\end{document}